\documentclass{amsart}
\usepackage[foot]{amsaddr}
\usepackage[activeacute,english]{babel}
\usepackage[latin1]{inputenc}
\usepackage{amsmath}                        
\usepackage{amsthm}                         
\usepackage{amsfonts}                       
\usepackage{amssymb}
\usepackage{dsfont}
\usepackage{babel}
\usepackage{graphicx}
\usepackage{color}
\usepackage{colortbl}
\usepackage{dsfont}
\usepackage{anysize}                        
\usepackage{supertabular}
\usepackage{float}
\usepackage[body={212mm,265mm},left=3.5cm,right=2.5cm,top=3cm,bottom=3cm]{geometry}
\usepackage{setspace}
\usepackage{lscape}
\usepackage{cancel}
\usepackage{cases}
\usepackage{MnSymbol}
\usepackage{hyperref}

\def\BibTeX{{\rm B\kern-.05em{\sc i\kern-.025em b}\kern-.08em
    T\kern-.1667em\lower.7ex\hbox{E}\kern-.125emX}}

\newtheorem{thm}{Theorem}[section]
\newtheorem{cor}[thm]{Corollary}

\theoremstyle{definition}
\newtheorem{defn}[thm]{Definition}
\theoremstyle{remark}

\begin{document}
\title{Regularity of  viscous solutions for  a degenerate non-linear Cauchy problem} 

\author[1]{Eric Hernández Sastoque$^1$}
\address{$^1$Departamento de Matemáticas\\
	    Universidad de Magdalena\\
	    Santa Marta\\
	    Colombia}
\email{$^1$eric.hernandez@unimagdalena.edu.co}

\author{Christian Klingenberg$^2$}
\address{$^2$Department of Mathematics, W$\ddot{\text{a}}$rzburg University, Germany}
\email{$^2$klingenberg@mathematik.uni-wuerzburg.de}
\author{Leonardo Rendón$^3$}
\address{$^3$Departamento de  Matemáticas, Universidad Nacional de Colombia, Bogotá}
\email{$^3$lrendona@unal.eu.co}
\author{Juan C. Juajibioy$^4$}
\address{$^4$ Departamento de Ciencias Naturales y Exáctas, Fundación Universidad Autonoma  de Colombia, Bogotá}
\email{$^4$jcjuajibioyo@unal.edu.co}



\date{\today }

\begin{abstract}
We consider  the  Cauchy problem  for a class of nonlinear  degenerate
parabolic  equation with  forcing. By using the  vanishing viscosity method
we  obtain generalized solutions. We  prove some regularity results
about this generalized  solutions.

\end{abstract}

\maketitle
\section{Introduction}\noindent
We consider the Cauchy problem for the following  nonlinear
degenerate parabolic  equation with forcing
\begin{align}
&u_{t} = u\Delta u - \gamma|\nabla u|^{2}+f(t,u), \ (x,t)\in  \mathds{R}^{N}\times \mathds{R}^{+},\label{eq:1}\\ 
&u(x,0)=u_{0}(x)\in C(\mathds{R}^{N})\cap L^{\infty}(\mathds{R}^{N}) \label{eq:2},
\end{align}
where $\gamma$ is a non-negative constant. Equation (\ref{eq:1}) arisen in severals applications of
biology and  phisycs, see \cite{poros}, \cite{poros2}. Equation (\ref{eq:1}) is
of degenerate parabolic  type:  parabolicity it is loss at points  where $u=0$, see \cite{poros}, \cite{benedeto} for
a most  datailed  description. In \cite{Ber90} a weak solution for  the homogeneous equation (\ref{eq:1}) is  constructed
by using  the  vanishing  viscosity method, this method was introduced  by Lions and Crandall \cite{lions},  when they
studied the  existence  of solutions to Hamilton-Jacobi  equations
\[
 u_t+H(x,t,u,Du)=0
\]
and consists in view the  equation
(\ref{eq:1}) as the limit  for $\epsilon\to 0$ of the equation
\begin{equation}\label{vanish1}
u_{t} = \epsilon \Delta u +u\Delta u - \gamma|\nabla u|^{2}+f(t,u), 
\end{equation}
where $\epsilon$ is a small positive number. The  reguarity of  the weak solutions for  the  homogeneous Cauchy problem 
(\ref{eq:1}),(\ref{eq:2}) was  studied  by the  author in \cite{Lu01}.In this paper we
extend the  above results  for  the inhomogeneous case, this extension is interesting, from  physical  viewpoint, since the
equation (\ref{eq:1}) is  related  with non-equilibrium  process in poros media  due
to  external  forces. We obtain the  following main theorem,

\begin{thm}
 If $\gamma\geq \sqrt{2N}-1$, $|\nabla (u_0^{1+\frac{\alpha}{2}})|\leq M$,
 where $M$ is  a positive constant such  as
 \[
  \alpha^2+(\gamma+1)\alpha+\frac{N}{2}\leq 0,
 \]
then the  viscosity solutions of the Cauchy problem (\ref{eq:1}), (\ref{eq:2})
satisfies
\begin{equation}\label{eq:7}
 |\nabla(u^{1+\frac{\alpha}{2}})|\leq M.
\end{equation}

\end{thm}

\section{Preliminaries}
\begin{defn} A function   $u\in  L^{\infty}(\Omega)\bigcap L^{2}_{Loc}([0,+\infty);H^{1}_{Loc}(\mathds{R}^{N}))$, 
is called  a weak solution of  (\ref{eq:1}),(\ref{eq:2}) if it satisfies  the following conditions:
\begin{itemize}
 \item[(i)] $u(x,t)\geq 0$ , $a .e$ in $\Omega$.
 \item[(ii)] $u(x,t)$ satisfies the following relation
 \begin{equation} \label{DOS_3}
\int\limits_{{\mathds{R}^{N}}} u_{0}\psi(x,0)\,dx + \iint\limits_{{\Omega}}(u\psi_{t}-u\nabla u\cdot \nabla\psi-(1+\gamma)|\nabla u|^{2}\psi-f(t,u)\psi)\, dx \, dt = 0,
\end{equation}\noindent
for any $\psi\in C^{1,1}(\overline{\Omega})$ with compact support in $\overline{\Omega}$.
\end{itemize}
\end{defn}\noindent
For  the construction  of a weak solution to the Cauchy problem  (\ref{eq:1}),(\ref{eq:2}),
we use the viscosity method: we add  the   term  $\epsilon\Delta u$ in the  equation (\ref{eq:1}) and we consider  the following   Cauchy problem
\begin{align}
&u_t = u\Delta u - \gamma|\nabla u|^{2} + f(t,u)+\epsilon\Delta u,  \ u \in  \Omega \label{vanish:1},\\ 
&u(x,0) = u_{0}(x), x\in \mathbb{R}^{N}\label{vanish:2}
\end{align}
where $\gamma \geq 0$, the existence of  solutions is garanteed by the Maximum
principle and then  we investigate the convergence of the solutions when $\epsilon \to 0$, in  fact, we will show that
when $\epsilon \to 0$, $u^{\epsilon}$ converges  to  the weak solution of 
(\ref{eq:1}),(\ref{eq:2}), but to cost of the loss of the uniqueness.
\begin{defn} The weak solution for  the Cauchy  problem  (\ref{eq:1}),(\ref{eq:2}) constructed by the vanishing viscosity method is called viscosity solution.
\end{defn}
\section{Estimates of Hölder}\label{s_eh}\noindent
In this  section we begin by  collecting  some a priori estimates for the  function  $u$.
\begin{thm}
\noindent If  $\gamma\geq \sqrt{2N}-1$, the initial  data (\ref{eq:2}) satisifes  $|\nabla(u_{0}^{1+\frac{\alpha}{2}})|\leq M$, where $M$ is a positive constant,
$\alpha^{2}+(\gamma+1)\alpha+\frac{N}{2}\leq0$ and $f \in C^{1} (\mathds{R}^{+}\times\mathds{R})$ satisfies, $f\geq0$, $f_{u}\leq 0$ , then the viscosity solution $u(x,t)$ of Cauchy problem (\ref{eq:1}),(\ref{eq:2}) satisfies 
\[|\nabla(u^{1+\frac{\alpha}{2}})|\leq M, \  \text{in}  \ \overline{\Omega}.\]
\end{thm}
\begin{proof}

Let 
\begin{equation} \label{TRES_1}
w=\frac{1}{2}\sum_{i=1}^N u^{2}_{x_{i}}.
\end{equation}
Deriving with  respect $t$ in   (\ref{TRES_1}) and  replacing  in (\ref{eq:1}) we have 

\begin{equation}\label{neweq-1}
w_{t}=\sum_{i=1}^N u_{x_{i}}\Bigg[u_{x_{i}}\Delta u + u\Bigg(\sum_{j=1}^N u_{x_{i}x_{j}x_{j}}\Bigg)-2\gamma w_{x_{i}} + f_{u}u_{x_{i}}\Bigg].  \nonumber
\end{equation}
By other hand 
\begin{eqnarray}
\Delta w & = & \frac{1}{2}\sum_{j=1}^N \Bigg(\sum_{i=1}^N u^{2}_{x_{i}}\Bigg)_{x_{j}x_{j}}  \nonumber \\ \nonumber \\
& = & \frac{1}{2}\Bigg[\sum^{N}_{j=1}(2u_{x_{1}}u_{x_{1}x_{j}})_{x_{j}}+\sum^{N}_{j=1}(2u_{x_{2}}u_{x_{2}x_{j}})_{x_{j}}+\dots+\sum^{N}_{j=1}(2u_{x_{N}}u_{x_{N}x_{j}})_{x_{j}}\Bigg]
\nonumber \\ \nonumber \\
\Delta w & = & \sum^{N}_{i,j=1}u^{2}_{x_{i}x_{j}}+\sum^{N}_{i,j=1}u_{x_{i}}u_{x_{i}x_{j}x_{j}}, \label{deltaw}
\end{eqnarray}
thereby, 
\begin{eqnarray}
 w_{t} & = & 2w\Delta u + u\Delta w - u\sum_{i,j=1}^N u^{2}_{x_{i}x_{j}}-2\gamma \sum_{i=1}^N u_{x_{i}}w_{x_{i}}+2f_{u}w. \label{TRES_2}
 \end{eqnarray}
Set,
\begin{equation}\label{TRES_3}
 z=g(u)w.
\end{equation}
Deriving  two  times with  respect  $x_i$ in (\ref{TRES_3})   we have
\begin{align}
 w_{x_{i}}&=(g^{-1})_{x_{i}}z + g^{-1}z_{x_{i}}\label{TRES_4} \\
 w_{x_{i}x_{i}}&=(g^{-1})_{x_{i}x_{i}}z+2(g^{-1})_{x_{i}}z_{x_{i}}+g^{-1}z_{x_{i}x_{i}}\label{TRES_5}.
\end{align}
From  equations (\ref{deltaw}),(\ref{TRES_4}), (\ref{TRES_5}) we have that,

\begin{equation}
 \Delta w=\sum^{N}_{i=1}w_{x_{i}x_{i}}=\sum^{N}_{i=1}\Big[(g^{-1})_{x_{i}x_{i}}z+2(g^{-1})_{x_{i}}z_{x_{i}}+g^{-1}z_{x_{i}x_{i}}\Big], \nonumber
\end{equation}
Deriving two  times with  respect $x_i$ in (\ref{TRES_3}) we have
\begin{align}
 (g^{-1}(u))_{x_{i}}&=-g^{-2}g^{'}u_{x_{i}} \\
(g^{-1}(u))_{x_{i}x_{i}} & = 
 \bigg(\frac{2g^{'2}-gg^{''}}{g_{4}}\bigg)gu^{2}_{x_{i}}-\frac{g^{'}}{g^{2}}u_{x_{i}x_{i}}, 
\end{align}
then,
\begin{eqnarray}
\Delta w & = & \bigg(\frac{2g^{'2}-gg^{''}}{g^{4}}\bigg)g\sum^{N}_{i=1}u^{2}_{x_{i}}z-\frac{g^{'}}{g^{2}}\sum^{N}_{i=1}u_{x_{i}x_{i}}z-2g^{-2}g^{'}\sum^{N}_{i=1}u_{x_{i}}z_{x_{i}}+g^{-1}\sum^{N}_{i=1}z_{x_{i}x_{i}}
\nonumber \\ \nonumber \\
& = & g^{-1}\sum^{N}_{i=1}z_{x_{i}x_{i}}-2g^{-2}g^{'}\sum^{N}_{i=1}u_{x_{i}}z_{x_{i}}+2\bigg(\frac{2g^{'2}-gg^{''}}{g^{4}}\bigg)gwz-\frac{g^{'}}{g^{2}}z\sum^{N}_{i=1}u_{x_{i}x_{i}} \nonumber \\ \nonumber \\
\Delta w & = & g^{-1}\Delta z -2g^{-2}g^{'}\sum^{N}_{i=1}u_{x_{i}}z_{x_{i}}+2\bigg(\frac{2g^{'2}-gg^{''}}{g^{4}}\bigg)z^{2}-\frac{g^{'}}{g^{2}}z\Delta u. \label{TRES_6} \end{eqnarray}
From (\ref{TRES_2}), (\ref{TRES_3}), (\ref{TRES_4}), (\ref{TRES_6}), we obtain


\begin{equation}\label{TRES_7}
 \begin{split}
   z_{t} = & u\Delta z - (2g^{-1}ug^{'}+2\gamma)\sum^{N}_{i=1}u_{x_{i}}z_{x_{i}}+(2f_{u}+g^{'}g^{-1}f(t,u))z \\
   &+ \left(\frac{4ug^{'2}}{g^{3}}-\frac{2ug^{''}}{g^{2}}+\frac{2\gamma g^{'}}{g^{2}}\right)z^{2}+2z\Delta u - ug(u)\sum^{N}_{i,j=1}u^{2}_{x_{i}x_{j}}.
 \end{split}
\end{equation}
By choosing  $g(u)=u^{\alpha}$, and since  
\begin{equation}\label{TRES_9}
 \sum^{N}_{i,j=1}u^{2}_{x_{i}x_{j}}\geq\frac{1}{N}(\Delta u)^{2},
\end{equation}
replacing $g$ in  (\ref{TRES_7}),(\ref{TRES_9}) we have
\begin{equation}\label{TRES_10}
  \begin{split}
   z_{t} \leq & u\Delta z - 2(\alpha + \gamma)\sum^{N}_{i=1}u_{x_{i}}z_{x_{i}}+(2f_{u}+\alpha u^{-1}f(t,u))z  \\
   &+ 2\alpha(\alpha+1+\gamma)u^{-\alpha-1}z^{2}+2z\Delta u - \frac{u^{\alpha+1}}{N}(\Delta u)^{2}.
  \end{split}
 \end{equation}
For $\gamma \geq \sqrt{2N}-1$, if $\alpha$ satisfies
\begin{equation}\label{TRES_11}
  \alpha^{2}+(\gamma+1)\alpha + \frac{N}{2}\leq 0,
\end{equation}
 where $\alpha^{2}+(\gamma+1)\alpha\leq-\frac{N}{2}$,
 then,
\begin{equation}\label{TRES_12}
  2\alpha(\alpha+\gamma+1)u^{-\alpha-1}z^{2}+2z\Delta u - \frac{u^{\alpha+1}}{N}(\Delta u)^{2} \leq 0.
\end{equation}
Therefore from (\ref{TRES_10}) and (\ref{TRES_12}) we have
\begin{equation}\label{TRES_13}
  z_{t} \leq u\Delta z - 2(\alpha + \gamma)\sum^{N}_{i=1}u_{x_{i}}z_{x_{i}}+(2f_{u}+\alpha u^{-1}f(t,u))z.
\end{equation}
By an application of the maximum principle  in (\ref{TRES_13}) we have 
\begin{equation}
  |z|_{\infty}\leq|z_{0}|_{\infty}. \nonumber
\end{equation}
Now, from (\ref{TRES_1}), (\ref{TRES_3}), with $g(u)=u^{\alpha}$, since  the initial  data (\ref{eq:2}) satisifes 
\[\newline|\nabla(u_{0}^{1+\frac{\alpha}{2}})|\leq M,\]
with $M$ a positive constant and $\alpha$ satisfies (\ref{TRES_11}), we have

\begin{eqnarray*}
 |\nabla(u^{1+\frac{\alpha}{2}})|^{2} & = & \left|\sum^{N}_{i=1}(u^{1+\frac{\alpha}{2}})_{x_{i}}e_{i}\right|^{2}  \\ \\
 & = & \sum^{N}_{i=1}\left[(u^{1+\frac{\alpha}{2}})_{x_{i}}\right]^{2} \\ \\
 & = & \sum^{N}_{i=1}\left[\left(1+\frac{\alpha}{2}\right)u^{\frac{\alpha}{2}}u_{x_{i}}\right]^{2} \\ \\
 & = & \left(1+\frac{\alpha}{2}\right)^{2}u^{\alpha}\sum^{N}_{i=1}u^{2}_{x_{i}} \\ \\
 & = & 2\left(1+\frac{\alpha}{2}\right)^{2}u^{\alpha}w \\ \\
 & = & 2\left(1+\frac{\alpha}{2}\right)^{2}z,
\end{eqnarray*}
 therefore
\begin{displaymath}
 |\nabla(u^{1+\frac{\alpha}{2}})| \leq M.
\end{displaymath}
\end{proof}
\section{Hölder Continuity of $u(x,t)$}\label{s_hc}\noindent
Now, using  Theorem 3.1, we have the following corollary about  the regularity of the viscosity solution $u(x,t)$ to the  Cauchy problem  (\ref{eq:1}),(\ref{eq:2}).
\begin{cor}
\noindent Let $f$ be a continuous fuctions such that
 \[|f(t,w)| \leq k|w|^{m},\]
where $w$ is a real value function and $m$, $k$ non-negative constants.  Under  conditions  of the Theorem 3.1 the viscosity
solution $u(x,t)$ of the Cauchy problem (\ref{eq:1}), (\ref{eq:2}) is   Lipschitz continuous with  respect to $x$ and locally Hölder  continuous with exponent $\frac{1}{2}$ with  respect to  $t$ in $\overline{\Omega}$.
\end{cor}
\begin{proof}
From Theorem 3.1  there exists  $\alpha \in \mathds{R}$ with $\alpha^{2}+(\gamma+1)\alpha+\frac{N}{2}\leq0$, with  $\alpha<0$, or,
\begin{displaymath}
 -\frac{\sqrt{(\gamma+1)^{2}-2N}}{2}-\frac{\gamma+1}{2}\leq\alpha\leq-\frac{\gamma+1}{2}+\frac{\sqrt{(\gamma+1)^{2}-2N}}{2}<0.
\end{displaymath}
 Since $\alpha<0$,  taking $\alpha\neq-2$, we have  the  estimate,
\begin{eqnarray*}
 |\nabla(u^{1+\frac{\alpha}{2}})| & = & \Big|(1+\frac{\alpha}{2})u^{\frac{\alpha}{2}}\nabla u\Big| \\ \\
 & = & \Big|1+\frac{\alpha}{2}\Big|u^{\frac{\alpha}{2}}|\nabla u|\leq M.
\end{eqnarray*}
Now, as $u\geq0$, we have that
\begin{equation}\label{TRES_14}
  |\nabla u|\leq\Big|1+\frac{\alpha}{2}\Big|^{-1}u^{-\frac{\alpha}{2}}M\leq M_{1} \textrm{ in }\overline{\Omega},
\end{equation}
since $u$ is bounded.\\
Using  the value mean theorem  we have
\begin{equation}\label{TRES_15}
 u(x_{1},t)-u(x_{2},t)= \nabla u(x_{1}+\theta(x_{2}-x_{1}),t)\cdot(x_{1}-x_{2}),
\end{equation}
for any $\theta \in (0,1)$.
From (\ref{TRES_14}), (\ref{TRES_15}) we have,
\begin{eqnarray*}
 |u(x_{1},t)-u(x_{2},t)| & \leq & |\nabla u(x_{1}+\theta(x_{2}-x_{1}),t)||x_{1}-x_{2}| \\ \\
 & \leq & M_{1}|x_{1}-x_{2}|,\qquad\forall(x_{1},t),(x_{2},t)\in\Omega.
\end{eqnarray*}
Therefore  $u(x,t)$ is a Lipschitz continuous with respect to the spatial variable.\\[0.2in]
For  Hölder continuity of $u(x,t)$ with respect to the temporary variable, we are going to use the ideas  developed  in \cite{Gil76}.
Let  $u_{\epsilon}(x,t) \in C^{2.1}(\Omega)\bigcap C(\overline{\Omega})\bigcap L^{\infty}(\Omega)$  the classical solution 
to the Cauchy problem problem (\ref{eq:1}), (\ref{eq:2}), namely,
\begin{displaymath}
 \begin{cases} u_{t}  =  u\Delta u - \gamma |\nabla u|^{2} + f(t,u) & \text{ in $\Omega$} \\ u(x,0) =  u_{0}(x) + \epsilon & \text{on $\mathds{R}^{N},$}\end{cases}
\end{displaymath}
We have that
\begin{eqnarray*}
 \Bigg|\nabla(u_{0}+\epsilon)^{1+\frac{\alpha}{2}}\Bigg| & = & \Bigg|\Big(1+\frac{\alpha}{2}\Big)(u_{0}+\epsilon)^{\frac{\alpha}{2}}\nabla u_{0}\Bigg| \\ \\
 & \leq & \Big|1+\frac{\alpha}{2}\Big|(u_{0})^{\frac{\alpha}{2}}|\nabla u_{0}| \\ \\
 & = & \Bigg|\nabla \Big(u_{0}^{1+\frac{\alpha}{2}}\Big)\Bigg|. \\ \\
 & \leq & M,
\end{eqnarray*}
Then, the  conditions  of Theorem 3.1  holds. 
Thereby
\begin{displaymath}
 \Big|\nabla (u_{0}+\epsilon)^{1+\frac{\alpha}{2}}\Big|\leq M.
\end{displaymath}
Since $u_{\epsilon}$ is a classical solution, $u$ is also  a weak solution of the Cauchy problem (\ref{vanish:1}), (\ref{vanish:2}). Hence, using the  same  arguments in the  proof of Theorem 3.1,  we have that $u_{\epsilon}$ is a Lipschitz continuous with respect to the spatial variable, with constant $M$, namely 
\begin{equation}\label{TRES_16}
 |u_{\epsilon}(x_{1},t)-u_{\epsilon}(x_{2},t)| \, \leq \,M |x_{1}-x_{2}| \,\,\,\, \forall \,\, (x_{1},t), (x_{2},t) \in \Omega.
\end{equation}
\bigskip
\noindent Now, let $z=u_{\epsilon}$ be, then we have,

\begin{displaymath}
 z_{t}=u_{\epsilon_{t}}= u_{\epsilon}\Delta u_{\epsilon} - \gamma |\nabla u_{\epsilon}|^{2}+f(t,u_{\epsilon})
\end{displaymath}

\bigskip
\noindent or,

\begin{equation}\label{TRES_17}
 u_{\epsilon}\Delta z - z_{t} = \gamma |\nabla u_{\epsilon}|^{2} - f(t,u_{\epsilon}) \textrm{ in } \Omega.
\end{equation}

\bigskip
\noindent Using (\ref{TRES_17}) we have  that for all $T>0$,$R>0$, $z$ satisfies the  equation

\begin{equation}\label{TRES_18}
 u_{\epsilon}\Delta z - z_{t} = \gamma |\nabla u_{\epsilon}|^{2} - f(t,u_{\epsilon}) \textrm{ in } B_{2R}(0)\times (0,T],
\end{equation}

\bigskip
\noindent where $B_{2R}(0)$ is the open ball  centered in 0, with  radius 2$R$ in $\mathds{R}^{N}$. Noticing that $u_{\epsilon}\in C^{2.1}\Big(B_{2R}(0))\times (0,T]\Big)$.

\bigskip
\noindent Now, since $u_{\epsilon}$ and $\nabla u_{\epsilon}$ are  bounded in $\overline{B_{2R}(0)}\times (0,T]$, there exists a constant  $\mu>0$, such that

\begin{displaymath}
 \sum_{i=1}^{N}u_{\epsilon}(x,t)=N u_{\epsilon}(x,t)\leq\mu,
\end{displaymath}

\begin{displaymath}
 \gamma |\nabla u_{\epsilon}(x,t)|\leq\mu,\qquad\forall(x,t)\in B_{2R}(0)\times (0,T],
\end{displaymath}

\bigskip
\noindent and

\begin{displaymath}
 f(t,u_{\epsilon})\leq \mu.
\end{displaymath}

\bigskip
\noindent From (\ref{TRES_16}), we have also

\begin{displaymath}
 |z(x_{1},t)-z(x_{2},t)| \leq M |x_{1}-x_{2}| \qquad\forall(x,t)\in B_{2R}(0))\times (0,T].
\end{displaymath}
\noindent In acording with  \cite{Gil76},  there exists a positive constant $\delta$ (which  depends only of $\mu$ and $R$)
and a positive constant $K$, which  depends only of $\mu$, $R$ and $M$, such that

\begin{displaymath}
 |z(x,t)-z(x,t_{0})| \leq K |t-t_{0}|^{\frac{1}{2}},
\end{displaymath}

\bigskip
\noindent for all $(x,t), (x,t_{0})\in B_{R}(0)\times (0,T]$ with $|t-t_{0}|<\delta$.

\bigskip
\noindent That  is,

\begin{displaymath}
 |u_{\epsilon}(x,t)-u_{\epsilon}(x,t_{0})| \leq K |t-t_{0}|^{\frac{1}{2}},
\end{displaymath}

\bigskip
\noindent for all $(x,t), (x,t_{0}) \in B_{R}(0)\times (0,T]$ with $|t-t_{0}|<\delta$.

\bigskip
\noindent Whenever  $K$ is independent of $\epsilon$, taken  $\epsilon\searrow0$, we obtain

\begin{displaymath}
 |u(x,t)-u(x,t_{0})| \leq K |t-t_{0}|^{\frac{1}{2}},
\end{displaymath}

\bigskip
\noindent for all $(x,t), (x,t_{0}) \in B_{R}(0)\times (0,T]$ with $|t-t_{0}|<\delta$.
\nocite{Ber90, Eva98, Fri64, Fri78, Kes89, Lad68, Lu99, Lu00, Lu01, Qia99, Pro84}
\end{proof}
\bibliographystyle{amsplain}
\bibliography{bibliografia}

\providecommand{\bysame}{\leavevmode\hbox to3em{\hrulefill}\thinspace}
\providecommand{\MR}{\relax\ifhmode\unskip\space\fi MR }
\providecommand{\MRhref}[2]{%
  \href{http://www.ams.org/mathscinet-getitem?mr=#1}{#2}
}
\providecommand{\href}[2]{#2}
\begin{thebibliography}{10}

\bibitem{benedeto}
Emmanuele DiBenedetto, \emph{Degenerate parabolic equations}, Springer-Verlag,
  New York, Heidelberg, Berlin, 1993.

\bibitem{Eva98}
Lawrence~C. Evans, \emph{Partial differential equations}, American Mathematical
  Society, Graduate Studies In Mathematics. Rhode Island,, 1998.

\bibitem{Fri64}
Avner Friedman, \emph{Partial differential equations of parabolic type},
  Englewood Cliffs, N.J., Prentice-Hall Inc, 1964.

\bibitem{Fri78}
John Fritz, \emph{Differential equations}, Springer-Verlag, New York,
  Heidelberg, Berlin, 1978.

\bibitem{Gil76}
B.H. Gilding, \emph{Hölder continuity of solutions of parabolic equations}, J.
  Landon Math. Soc. 13, 103-106, 1976.

\bibitem{Kes89}
S.~Kesavan, \emph{Topics in functional analysis and applications}, John Wiley
  \& Sons. New York, 1989.

\bibitem{Lad68}
O.A. Ladysenskaya, V.A. Solonnikov, and Ural´ceva N.N, \emph{Linear and
  quasilinear equations of parabolic type}, Amer.Math.Soc. Transl, 1968.

\bibitem{Lu99}
Yun-Guang Lu, \emph{Hölder estimates of solutions to some doubly nonlinear
  degenerate parabolic equations}, Comm. Partial Differential Equations 24, no.
  5-6, 895--913.5656, 1999.

\bibitem{Lu01}
Yun~Guang Lu and Liwen Qian, \emph{Regularity of viscosity solutions of a
  degenerate parabolic equation}, American Mathematical Society, volume 130,
  number 4. Pages 999-1004, 2001.

\bibitem{lions}
Pierre-Louis~Lions Michael G.~Crandall, \emph{Viscosity solutions of
  hamilton-jacobi equations}, Transactions of the American Mathematical Society
  (1983).

\bibitem{Ber90}
Maura~Ughi Michiel~Bertsch, Roberta Dal~Passo, \emph{Discontinuous viscosity
  solutions of a degenerate parabolic equation}, Trans Amer. Math. Soc. 320,
  no. 2, 779-798, 1990.

\bibitem{poros2}
A.~Mikelic M.S~Espedal, A.~Fasano, \emph{Filtration in porous media and
  industrial applications}, Springer-Verlag, New York, Heidelberg, Berlin,
  2000.

\bibitem{Pro84}
Murray~H. Protter and Hans~F.W. Weinberger, \emph{Maximum principles in
  differential equations}, Springer-Verlag, New York, Heidelberg, Berlin, 1984.

\bibitem{Qia99}
Liwen Quian and Wentao Fan, \emph{Hölder estimate of solutions of some
  degenerate parabolic equations}, Acta Math. Sci. (English Ed.) 19, no. 4,
  463--468, 1999.

\bibitem{poros}
Juan~Luis Vazquéz, \emph{The porous medium equation, mathematical theory},
  Oxford Science Publications, 2007.

\end{thebibliography}

\end{document}